\documentclass[11.5pt]{amsart}
\usepackage{euscript,amsmath,amssymb,amsfonts,graphicx,bm,amsthm,mathtools}
\usepackage{cite}
\usepackage{paralist}
\usepackage{graphics} 
\usepackage[caption=false]{subfig}
\usepackage{latexsym,enumerate} 
\usepackage{tikz}
\usepackage{hyperref}
\hypersetup{urlcolor=blue, citecolor=red}

\usepackage{pgfplots}
\usepackage{soul,xcolor}

\def\E{\mathbb{E}}
\def\P{\mathbb{P}}

\newtheorem{theorem}{Theorem}
\newtheorem{lemma}[theorem]{Lemma}

\newtheorem{corollary}[theorem]{Corollary}
\theoremstyle{plain}
\theoremstyle{remark}

\theoremstyle{definition}

\usepackage[noabbrev]{cleveref}

\usepackage{abstract}

\usepackage{adjustbox}
\usepackage{comment}

\begin{document}
\title{Cover times with stochastic resetting}
\author{Samantha Linn}
\thanks{SL (corresponding author): Department of Mathematics, University of Utah, Salt Lake City, UT 84112 USA. (\texttt{slinn.math@gmail.com}).}
\author{Sean D. Lawley}
\thanks{SDL: Department of Mathematics, University of Utah, Salt Lake City, UT 84112 USA. (\texttt{lawley@math.utah.edu}).}
\maketitle

\begin{abstract}
Cover times quantify the speed of exhaustive search. In this work, we approximate the moments of cover times of a wide range of stochastic search processes in $d$-dimensional continuous space and on an arbitrary discrete network under frequent stochastic resetting. These approximations apply to a large class of resetting time distributions and search processes including diffusion, run-and-tumble particles, and Markov jump processes. We illustrate these results in several examples; in the case of diffusive search we show that the errors of our approximations vanish exponentially fast. Finally we derive a criterion for when endowing a discrete state search process with minimal stochastic resetting reduces the mean cover time.
\end{abstract}
\vspace{.5cm}
\textbf{Cover times measure the time taken by a stochastic searcher to find multiple targets. Cover times thus naturally extend first passage times (FPTs), which describe the time to find a single target. Recent work has shown that stochastic resetting, wherein a searcher randomly resets in space, can reduce mean FPTs. In this work, we study cover times with stochastic resetting. In particular, we approximate statistics of cover times under frequent stochastic resetting. Using recent results on FPTs, we find that such cover times resemble FPTs to the most distant targets. Quantitatively, these cover times are approximately exponentially distributed with a rate depending on the resetting rate and the distant target positions. We illustrate our results through numerous examples in $d$-dimensional continuous space and on a discrete network, which reveal that stochastic resetting can reduce mean cover times.}

\section{Introduction}
An exhaustive search is one in which an entire state space is explored by a searcher. Such search processes are ubiquitous in technology and the natural world, ranging from computer search algorithms \cite{Vergassola2007} to animal foraging \cite{VISWANATHAN2008133, benichou2011rev} or pathogen hunting \cite{heuze2013}. In some cases the state space of interest, or target region, may be a proper subset of the searcher domain but it may also be the entire domain. Whatever the target region, we call the time taken to complete such a search its \emph{cover time}.

To be precise, denote the position of a stochastic process in a domain $M\subseteq\mathbb{R}^d$ by $X = \{X(t)\}_{t\geq 0}$. If this stochastic `searcher' has detection radius $R>0$, then the region it explores by time $t> 0$ is
\begin{align*}
    S_0(t) := \cup_{s=0}^t B(X(s),R) \subseteq M
\end{align*}
where $B(X,R)$ is the ball of radius $R>0$ centered at $X\in M$. Hence, the cover time of a given target region $U_T\subseteq M$ is
\begin{align*}
    T_0 := \inf\{ t\geq 0 : U_T \subseteq S_0(t) \}.
\end{align*}

The bulk of previous work on cover times concerns stochastic search processes on discrete state spaces \cite{Aldous1983, Aldous1989,Kahn1989,Belius2013,CT_chupeau,Maier2017,Weng2017_chaos, Maziya2020,Campos2021,Dong2023, Martinez2025}. A seminal result on this topic is that, in the limit of a large target region, cover times of random walks follow a Gumbel distribution \cite{Belius2013}. There has also been notable progress characterizing cover times on continuous state spaces by making use of various asymptotic limits. For instance, it has been shown that cover times of Brownian motion diverge as the detection radius goes to zero \cite{Dembo2003}, 
\begin{align*}
    T_0 \sim t_d\quad\text{almost surely as }R\to 0,
\end{align*}
where $t_d$ depends on the searcher diffusivity, $D>0$, the domain volume, $|M|$, and the spatial dimension, $d>0$,
\begin{align*}
    t_d = 
    \begin{cases}
        \frac{1}{\pi}\frac{|M|}{D} (\ln(1/R))^2 \quad &\text{if }d=2,\\
        \frac{d\Gamma(d/2)R^{2-d}}{2(d-2)\pi^{d/2}} \frac{|M|}{D} \ln(1/R) \quad &\text{if }d\geq 3.
    \end{cases}
\end{align*}
Throughout, $f\sim g$ denotes $f/g\to 1$. 

In more recent work, others have considered cover times of multiple simultaneous searchers of fixed detection radius in the many searcher limit. In continuous space, the $m$th moment of the cover time of $U_T\subseteq M$ for $N\gg 1$ Brownian searchers, $T(N)\geq 0$, satisfies
\begin{align*}
    \E((T(N))^m) \sim \Big(\frac{L^2}{4D\ln N}\Big)^m\quad\text{as }N\to\infty,
\end{align*}
where $D>0$ denotes searcher diffusivity and $L>0$ denotes the shortest distance the searcher must travel to reach the furthest part of the target region \cite{CT_diff}. Moreover, in discrete space, these moments are given by
\begin{align*}
    \E((T(N))^m) \sim \frac{\kappa_m}{N^{m/h^*}} \quad\text{as }N\to\infty,
\end{align*}
where $h^*\geq 1$ is the fewest steps the searcher must take to reach the furthest part of the target region and $\kappa_m$ depends on the jump rates along this path \cite{CT_N_discrete}. An important observation in the $N\gg 1$ searcher case is that the cover times can be closely approximated with merely some information about the geodesic paths from the initial position of the searcher to the furthest part(s) of the target region. The same can be said of the work herein.

In this paper we consider the cover time of a single stochastic searcher that undergoes stochastic resetting to its initial position. In continuous space the searcher has a fixed detection radius and in discrete space the searcher detects only the node it occupies. To our knowledge, the only other investigation of cover times under stochastic resetting is the work of Colombani et al \cite{Colombani_2023} who studied how resetting can be used to reduce mean cover times (MCTs) of a self-avoiding random walker on a discrete network.

Not only is stochastic resetting a useful formalism to model physical phenomena \cite{Brockwell_Gani_Resnick_1982,Visco2010,Roldan_16}, it can also expedite certain search processes \cite{optMFPT,Stanislavsky2021,calvert2021,debruyne2023}. For instance, consider a Brownian searcher with diffusivity $D>0$ on the half line. It is well-known that without stochastic resetting the mean first passage time (MFPT) to the origin is infinite.
\begin{figure}[h!]
  \centering
\includegraphics[width=8cm]{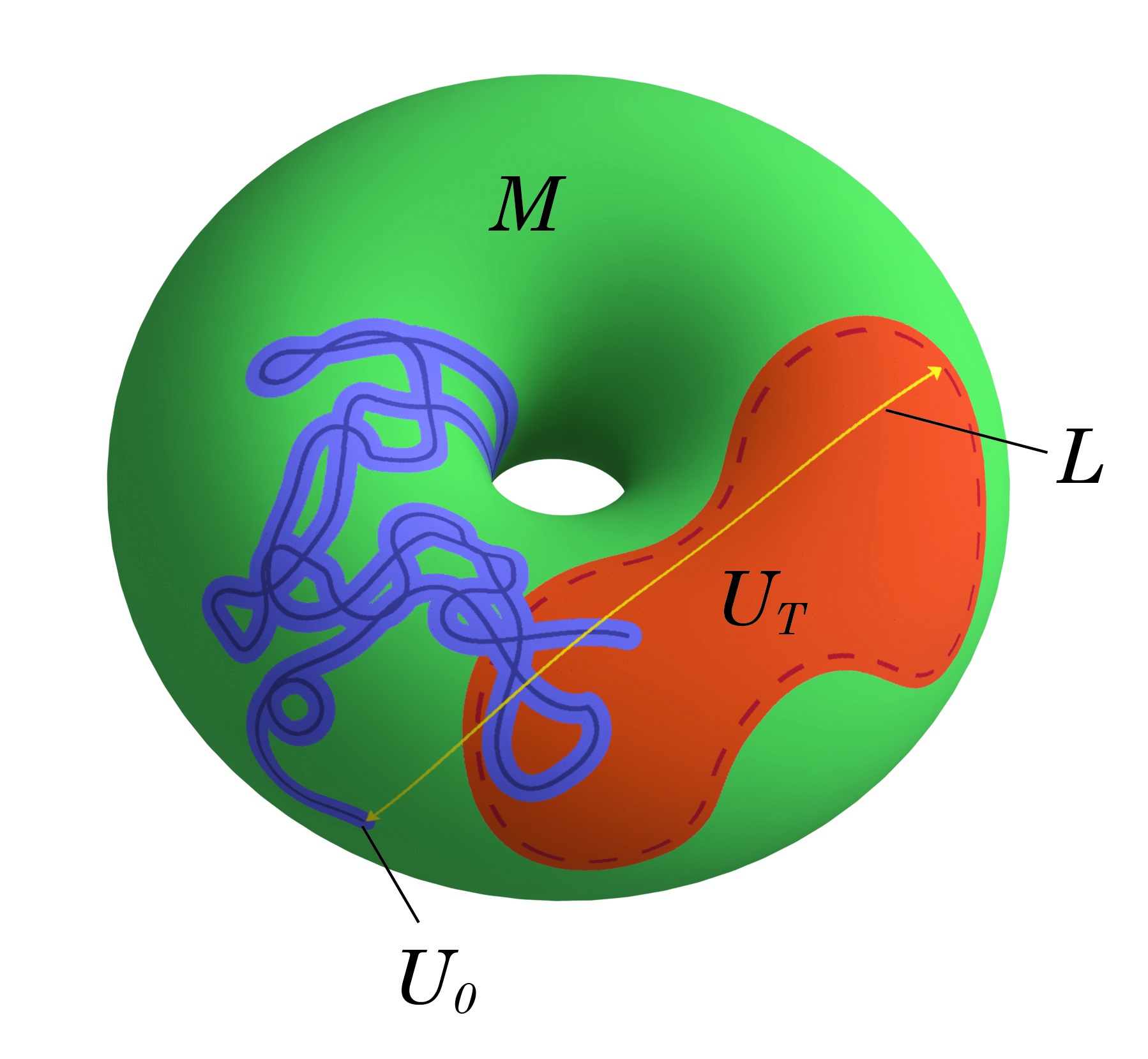}
  \caption{Representative illustration of a spatial domain $M$ (pictured, a 2D torus) with a stochastic searcher that starts at $U_0$. Dark purple indicates the searcher path and light purple indicates the corresponding region detected. Orange denotes the target region ($U_T$) and yellow denotes the shortest path (length $L$) that the searcher must travel to cover the farthest part of the target.}
  \label{fig0}
\end{figure}
However, if this searcher instantaneously resets to its initial position according to an exponential random variable with rate $r>0$, then the MFPT to the origin is finite. That is, denoting its position at time $t> 0$ by $X=\{X(t)\}_{t\geq 0}$ and supposing $X(0) = x_0 > 0$, we have that
\begin{align*}
    \mathbb{E}[\text{inf} \{ t\geq 0 : X(t) \leq 0 \}] = \frac{1}{r}(e^{x_0\sqrt{r/D}} - 1) < \infty,
\end{align*}
which attains a minimum for a strictly positive resetting rate \cite{Evans2020}.

Similar to MFPTs, we show that stochastic resetting can be used to reduce MCTs, though the resetting rates that achieve these reductions may differ between the two. Moreover, we show that the frequent resetting asymptotics of the moments of cover times in discrete and continuous space depend only on the geodesic path(s) from the initial (resetting) position of the searcher to the furthest part(s) of the target region. Our results hold for any fixed detection radius. While these results are exact only in the frequent resetting limit, we compute exactly the MCT with any resetting rate $r>0$ in the cases of a one-dimensional Brownian searcher and run-and-tumble particle.

These results are made precise in the rest of the paper, which is organized as follows. In section \ref{freq}, we consider the $d$-dimensional continuous space cover time problem with resetting. In section \ref{net}, we consider the cover time problem with resetting on an arbitrary discrete network. Section \ref{examples} illustrates these general results in several examples. We conclude by discussing the scope of the results and proposing directions of future work.

\section{Cover times for a frequently resetting searcher in $\mathbb{R}^d$} \label{freq}

\subsection{Definitions and main result} \label{def}
Now consider a stochastic searcher in the domain $M \subseteq \mathbb{R}^d$ equipped with a metric $L:M\times M \to [0,\infty)$ and with fixed detection radius $R>0$ which randomly resets to its initial position. The searcher may be Brownian, though other stochastic processes also satisfy the minimal assumptions imposed on the search process (see section \ref{examples}). Denoting the searcher position at time $t\geq0$ by $X_r(t)$, the region detected by the searcher at time $t\geq0$ is the closed ball of radius $R>0$ centered at $X_r(t)$,
\begin{align*}
    B(X_r(t),R) := \{ y\in M : L(X_r(t),y) \leq R\}.
\end{align*}
The subscript $r$ of $X_r$ denotes the resetting rate, which we formalize below. Denote by $U_0$ the region the searcher initially detects,
\begin{align*} 
    U_0 := B(X_r(0),R).
\end{align*}

We consider resetting processes, $\sigma>0$, that belong to the following family of random variables: Let $Y>0$ be a strictly positive random variable with unit mean,
\begin{align*}
    \E[Y] = 1.
\end{align*}
Moreover, suppose $Y>0$ has a finite moment generating function in a neighborhood of the origin,
\begin{align*}
    \E[e^{zY}] < \infty\quad\text{for all }z\in[-\delta,\delta]
\end{align*}
for some $\delta>0$. Define 
\begin{align} \label{sig}
    \sigma := Y/r
\end{align}
where we call $r>0$ the resetting rate. If $Y$ is a unit rate exponential random variable, then $\sigma$ is exponentially distributed with rate $r$.

The region covered by the resetting searcher up through time $t\geq 0$ is
\begin{align*}
    S_r(t) := \cup_{s=0}^t B(X_r(s),R) \subseteq M,
\end{align*}
and the cover time of the nonempty compact target region $U_T\subseteq M$ is
\begin{align*}
    T_r := \inf\{ t\geq 0 : U_T \subseteq S_r(t) \}.
\end{align*}
To exclude trivial cases, we assume $U_T\not\subseteq U_0$. See Figure \ref{fig0} for an illustration.

Denote the first hitting time of the resetting searcher to a set $\Omega\subseteq M$ with initial position $X_r(0)=x_0$ and resetting position $x_r$ by
\begin{align} \label{taur}
    \tau_r(\Omega|x_0,x_r) := \inf\{ t\geq 0 : X_r(t) \in\Omega\}.
\end{align}
If $x_0=x_r$, we abbreviate \eqref{taur} by $\tau_r(\Omega)$.

Assume that for any union of balls with strictly positive radii that does not contain $U_0$, the search time without resetting $\tau_0(\Omega)$ is strictly positive. If $\Omega = B(x,R)$, then $\tau_r(\Omega)$ is equivalent to
\begin{align*}
    \tau_r(B(x,R)) = \inf\{ t\geq 0 : X_r(t) \in B(x,R) \} = \inf\{ t\geq 0 : x\in S_r(t)\}.
\end{align*}
We remark that throughout this work $T$ denotes a cover time and $\tau$ denotes a hitting time.

Denote by $p_{\Omega}$ the probability of a stochastic searcher in $M\subseteq \mathbb{R}^d$ with initial position $X_0(0)$ hitting a union of balls with strictly positive radii $\Omega\subseteq M$ before time $\sigma>0$,
\begin{align} \label{pO}
    p_{\Omega} := \mathbb{P}(\tau_0(\Omega) < \sigma).
\end{align}
To exclude trivial cases, we assume $p_{\Omega}>0$ for all $r>0$ and $\Omega\subset M$.

A key assumption of this section is that for any two unions of balls with strictly positive radii $\Omega_0,\Omega_1\subset M$, we have the following,
\begin{align} \label{porder}
    L(\Omega_0,X_r(0)) > L(\Omega_1,X_r(0)) \quad\text{implies}\quad \lim_{r\to\infty} p_{\Omega_0}/p_{\Omega_1} = 0
\end{align}
where the distance between two sets $\Phi,\Psi\subset M$ is given by
\begin{align*}
    L(\Phi,\Psi) := \inf_{x\in\Phi,y\in\Psi} L(x,y).
\end{align*}
In words, \eqref{porder} means that a search is much more likely to hit a close target than a far target before a fast resetting time. This assumption is satisfied, for instance, by diffusive search with exponentially distributed resetting in which case $p_{\Omega}$ satisfies
\begin{align} \label{lnpO}
    \ln p_{\Omega} \sim -\sqrt{L(\Omega,X_r(0))^2r/D}\quad\text{as }r\to\infty
\end{align}
where $D>0$ is the diffusivity and $L$ is the standard Euclidean distance (see Theorem 2 in \cite{FPTuFSR}). See subsections \ref{ball} and \ref{torus} for specific examples of this behavior.

Finally, we assume that there exists a nonempty set $U_T^{\text{far}}\subseteq U_T$ such that if $x^*\in U_T^{\text{far}}$ and $y\in U_T$ then
\begin{align*}
    L(y,X_r(0)) \leq L(x^*,X_r(0)).
\end{align*}
In words, $U_T^{\text{far}}$ is the set of target points which are farthest from the initial searcher position.

Now we state the main result of this section:
\begin{theorem} \label{thm1}
    Under the assumptions subsection \ref{def},
    \begin{align*}
    \liminf_{r\to\infty} \mathbb{E}[(rp_{B(y,R)}T_r)^m] \geq m! \quad\text{as }r\to\infty
    \end{align*}
    for any $y\in U_T$ and integer $m\geq 1$. Further, for any $x^*\in U_T^{\text{far}}$,
    \begin{align*}
    \limsup_{r\to\infty} \mathbb{E}[(rp_{B(x^*,(1-\varepsilon)R)}T_r)^m] = 0 \quad\text{for all }\varepsilon \in(0,1).
    \end{align*}
\end{theorem}
Roughly speaking, Theorem \ref{thm1} gives the following bounds on the moments of the cover time in the frequent resetting regime,
\begin{align} \label{approx}
    \frac{m!}{(rp_{B(x^*,R)})^m} \lesssim \mathbb{E}[(T_r)^m] \lesssim \frac{m!}{(rp_{B(x^*,(1-\varepsilon)R)})^m}.
\end{align}
Hence the asymptotic behavior of $T_r$ is determined by the success probability associated with the furthest part of the target. The approximate behavior in \eqref{approx} is clear in the case of a one-dimensional Brownian motion as demonstrated by \eqref{etrrp} in section \ref{exact}. We note that Theorem \ref{thm1} still holds if the success probabilities are replaced by asymptotically equivalent quantities.

Finally observe that through \eqref{approx} we can approximate the `optimal' resetting rate that minimizes the MCT. Namely, by finding the zero of the $r$-derivative of the first relation in \eqref{approx} with $m=1$, we attain the following general approximation,
\begin{align} \label{ropt}
    \frac{1}{r_{\textup{opt}}} \approx - \frac{\textup{d}}{\textup{d}r_{\textup{opt}}} \ln p_{B(x^*,R)}(r_{\textup{opt}}).
\end{align}
In section \ref{examples} we compute the approximate optimum in various examples and compare them to exact values.

\subsection{Proof}
The proof of Theorem \ref{thm1} relies on the following lemma, which was proven in \cite{FPTuFSR} (see Theorem 1 in \cite{FPTuFSR}).
\begin{lemma} \label{FPT}
    Let $p_{\Omega}$ be as in \eqref{pO} and let $\Omega\subset M$ be any union of balls such that $X_r(0)\not\in \Omega$. Then for integers $m\geq 1$, we have
    \begin{align*}
        \E[(\tau_r(\Omega))^m] \sim \frac{m!}{(rp_{\Omega})^m} \quad\text{as }r\to\infty.
    \end{align*}
\end{lemma}

\begin{proof}[Proof of Theorem \ref{thm1}]
We start with the lower bound: Let $y\in U_T$. The survival probability of $T_r$ is that of the union of hitting times to each point in the target region, which itself has a natural lower bound,
\begin{align} \label{lb}
    \mathbb{P}(T_r > t) = \mathbb{P}(\cup_{x\in U_T} \{ \tau_r(B(x,R)) > t\}) \geq \mathbb{P}(\tau_r(B(y,R)) > t).
\end{align}
Hence, substituting $t$ for $z^{1/m}$ with $m\in \mathbb{N}\backslash \{0\}$ in \eqref{lb} and integrating over $z\geq 0$ bounds the moments of $T_r$ from below,
\begin{align*}
    \mathbb{E}[(T_r)^m] &= \int_0^{\infty} \mathbb{P}(T_r > z^{1/m})\,\textup{d}z\\ &\geq \int_0^{\infty} \mathbb{P}(\tau_r(B(y,R)) > z^{1/m})\,\textup{d}z = \mathbb{E}[(\tau_r(B(y,R)))^m].
\end{align*}
By Lemma \ref{FPT} we find
\begin{align*}
    \liminf_{r\to\infty} \mathbb{E}[(rp_{B(y,R)}T_r)^m] \geq \lim_{r\to\infty} \mathbb{E}[(rp_{B(y,R)}\tau_r(B(y,R)))^m] = m!.
\end{align*}
Hence
\begin{align*}
    \liminf_{r\to\infty} \mathbb{E}[(rp_{B(y,R)}T_r)^m] \geq m!.
\end{align*}

Determining the limit supremum is more involved: Let $\varepsilon \in(0,1)$ and $x^*\in U_T^{\textup{far}}$, and cover $U_T$ by a finite set of $\varepsilon R$-radius balls with centers $\{x_1,\dots,x_k\}\in U_T$,
\begin{align*}
    \cup_{i=1}^k B(x_i,\varepsilon R) \supseteq U_T.
\end{align*}
With this cover, one can show \cite{CT_diff} that
for any $i\in\{1,\dots,k\}$,
\begin{align*}
    B(x_i,\varepsilon R) \subset B(y_i,R)\quad\text{if}\quad y_i\in B(x_i,(1-\varepsilon)R).
\end{align*}
Hence
\begin{align} \label{PTr}
\begin{split}
    \mathbb{P}(T_r > t) &= \mathbb{P}(\cup_{x\in U_T} \{ \tau_r(B(x,R)) > t\})\\
    &\leq \mathbb{P}(\cup_{i=1}^k \{ \tau_r(B(x_i,(1-\varepsilon)R)) > t\}).
\end{split}
\end{align}
Using the inclusion-exclusion principle \cite{durrett2019} we rewrite the rightmost expression in \eqref{PTr},
\begin{align} \label{PTr_cont}
\begin{split}
    \mathbb{P}(&\cup_{i=1}^k \{ \tau_r(B(x_i,(1-\varepsilon)R)) > t \})\\
    &= \sum_{j=1}^k \Big((-1)^{j-1} \sum_{\substack{I\subset \{ 1,\dots,k \} \\ |I|=j }} \mathbb{P}(\cap_{i\in I} \{ \tau_r (B(x_i,(1-\varepsilon)R)) > t \})\Big)\\
    &= \sum_{j=1}^k \Big((-1)^{j-1} \sum_{\substack{I\subset \{ 1,\dots,k \} \\ |I|=j }} \mathbb{P}( \min_{i\in I} \{ \tau_r (B(x_i,(1-\varepsilon)R)) \} > t )\Big).
\end{split}
\end{align}
Hence, substituting $t$ for $z^{1/m}$ with $m\in \mathbb{N}\backslash \{0\}$ in \eqref{PTr} and \eqref{PTr_cont}, and integrating over $z\geq 0$ bounds the moments of $T_r$ from above,
\begin{align*}
    \mathbb{E}[(T_r)^m] &\leq \sum_{j=1}^k \Big((-1)^{j-1} \sum_{\substack{I\subset \{ 1,\dots,k \} \\ |I|=j }} \int_0^{\infty} \mathbb{P}( \min_{i\in I} \{ \tau_r (B(x_i,(1-\varepsilon)R)) \} > z^{1/m} )\,\textup{d}z\Big) \\
    &= \sum_{j=1}^k \Big((-1)^{j-1} \sum_{\substack{I\subset \{ 1,\dots,k \} \\ |I|=j }} \mathbb{E}[( \min_{i\in I} \{ \tau_r (B(x_i,(1-\varepsilon)R)) \} )^m]\Big)\\
    &= \sum_{j=1}^k \Big((-1)^{j-1} \sum_{\substack{I\subset \{ 1,\dots,k \} \\ |I|=j }} \mathbb{E}[(\tau_r(\cup_{i\in I} B(x_i,(1-\varepsilon)R)))^m] \Big).
\end{align*}
Multiplying through by $(rp_{B(x^*,(1-2\varepsilon)R)})^m$ yields
\begin{align} \label{main}
\begin{split}
    \mathbb{E}[(r&p_{B(x^*,(1-2\varepsilon)R)}T_r)^m]\\ &\leq \sum_{j=1}^k \Big((-1)^{j-1} \sum_{\substack{I\subset \{ 1,\dots,k \} \\ |I|=j }} \mathbb{E}[(rp_{B(x^*,(1-2\varepsilon)R)}\tau_r(\cup_{i\in I} B(x_i,(1-\varepsilon)R)))^m] \Big)\\
    &= \sum_{j=1}^k \Big((-1)^{j-1} \sum_{\substack{I\subset \{ 1,\dots,k \} \\ |I|=j }} \Big(\frac{p_{B(x^*,(1-2\varepsilon)R)}}{p_{\cup_{i\in I} B(x_i,(1-\varepsilon)R)}}\Big)^m\\ &\qquad \qquad \qquad \qquad \times \mathbb{E}[(rp_{\cup_{i\in I} B(x_i,(1-2\varepsilon)R)}\tau_r(\cup_{i\in I} B(x_i,(1-\varepsilon)R)))^m] \Big).
\end{split}
\end{align}
Using Lemma \ref{FPT} and the assumption \eqref{porder} to take the limit as $r\to\infty$ yields
\begin{align} \label{last}
    \limsup_{r\to\infty} \mathbb{E}[(r&p_{B(x^*,(1-2\varepsilon)R)}T_r)^m] \leq 0.
\end{align}
Since the left-hand side of \eqref{last} is trivially non-negative, the proof is complete.
\end{proof}

\section{Cover times for a frequently resetting searcher on a network} \label{net}

\subsection{Definitions and main result} \label{def_net}
Now consider a stochastic searcher on a discrete state space $\Gamma$ that randomly resets to its initial position. In particular, let $X = \{X(t)\}_{t\geq 0}$ be a non-resetting continuous-time Markov chain on a finite or countably infinite state space $\Gamma$. We encode the dynamics of $X$ in its infinitesimal generator matrix $Q = \{q(i,j)\}_{i,j\in\Gamma}$ whose off-diagonal entries give the (non-negative) rate that $X$ jumps from $i\in\Gamma$ to $j\in\Gamma$ \cite{norris1998}. By construction the diagonal entries of $Q$ are non-positive and are such that $Q$ has zero row sums. We assume that $\sup_{i\in\Gamma} |q(i,i)|<\infty$ so that $X$ cannot take infinitely many jumps in finite time.

Let $X_r$ be the process $X$ modified to instantaneously reset to its initial position $i_0\in\Gamma$ according to $\sigma$ as in \eqref{sig}. Assume the detection `radius' of $X_r$ is zero; that is, the searcher detects only the node it occupies. Then the region covered by the resetting searcher up through time $t>0$ is
\begin{align*}
    S_r(t) := \cup_{s=0}^t X_r(s) \subseteq \Gamma
\end{align*}
and the corresponding cover time of the finite target region $U_T\subseteq \Gamma$ is
\begin{align*}
    T_r := \inf\{ t>0 : U_T \subseteq S_r(t) \}.
\end{align*}

As before, we define the first hitting time of $X_r(t)$ to a set $\Omega\subseteq \Gamma$,
\begin{align*}
    \tau_r(\Omega) := \inf\{ t\geq 0 : X_r(t) \in\Omega\},
\end{align*}
and the probability of a stochastic searcher hitting the set $\Omega\subseteq \Gamma$ before a resetting event at time $\sigma>0$,
\begin{align} \label{pOm}
    p_{\Omega} := \mathbb{P}(\tau_0(\Omega) < \sigma).
\end{align}

Following \cite{lawley2020networks}, define a path $\mathcal{P}$ of length $n\in\mathbb{N}$ from $i_1\in\Gamma$ to $i_{n+1}\in\Gamma$ to be a sequence of $n+1$ states in $\Gamma$,
\begin{align*}
    \mathcal{P} = (\mathcal{P}(1),\dots,\mathcal{P}(n+1)) = (i_1,\dots,i_{n+1}) \in \Gamma^{n+1},
\end{align*}
so that 
\begin{align} \label{path}
    q(\mathcal{P}(j),\mathcal{P}(j+1)) > 0,\quad\text{for }j\in\{1,\dots,n\}.
\end{align}
The condition in \eqref{path} says that there is a strictly positive probability of $X$ following the path $\mathcal{P}$. Naturally we assume there are paths from $i_0$ to each node in $U_T$.

For a path $\mathcal{P}\in\Gamma^{n+1}$, let $\lambda(\mathcal{P})$ be the product of the rates along the path,
\begin{align} \label{lamb}
    \lambda(\mathcal{P}) := \prod_{i=1}^n q(\mathcal{P}(i),\mathcal{P}(i+1)) > 0.
\end{align}
Let $h (\Gamma_0,\Gamma_1)\in\mathbb{N}$ denote the length of the geodesic path(s) from a set of nodes $\Gamma_0\subset\Gamma$ to a set of nodes $\Gamma_1\subset\Gamma$,
\begin{align*}
    h(\Gamma_0,\Gamma_1) := \min\{ n: \mathcal{P}\in\Gamma^{n+1}, \mathcal{P}(1)\in \Gamma_0, \mathcal{P}(n+1)\in \Gamma_1 \}.
\end{align*}
We define the set of all paths from $\Gamma_0$ to $\Gamma_1$ with this minimum length $h(\Gamma_0,\Gamma_1)$,
\begin{align} \label{S}
    \mathcal{S}(\Gamma_0,\Gamma_1) := \{ \mathcal{P}\in\Gamma^{n+1}: \mathcal{P}(1)\in \Gamma_0, \mathcal{P}(n+1)\in \Gamma_1, n = h(\Gamma_0,\Gamma_1) \}.
\end{align}
Define
\begin{align} \label{Lamb}
    \Lambda(\Gamma_1) := \sum_{\mathcal{P}\in\mathcal{S}(i_0,\Gamma_1)} \mathcal{P}(1) \lambda(\mathcal{P}).
\end{align}
In words, the quantity $\Lambda(\Gamma_1)$ in \eqref{Lamb} is the sum of the products of the jump rates along the shortest paths from $i_0$ to $\Gamma_1$. If there is only one such path, then $\Lambda(\Gamma_1)$ is merely the product of the jump rates along that path as in \eqref{lamb}.

Finally, let $h^*\geq 1$ denote the length of the geodesic path(s) between $i_0$ and the furthest node(s) in the finite target set $U_T$,
\begin{align} \label{dstar}
    h^* := \sup_{j\in U_T} h(i_0,j)
\end{align}
and let $U_T^{\text{far}}$ denote the furthest target point(s) in $U_T$,
\begin{align} \label{UTnet}
    U_T^{\text{far}} := \{ j\in U_T : h(i_0,j) = h^*\} \subseteq U_T.
\end{align}

With these definitions established, we now state the main result of this section:
\begin{theorem} \label{thm_net}
    Under the assumptions of section \ref{net}, for integers $m\geq 1$ we have
    \begin{align*}
    \mathbb{E}[(T_r)^m] \sim \frac{m!\,K_m}{r^{m(1-h^*)}} \quad\text{as }r\to\infty
    \end{align*}
    where
    \begin{align*}
        K_m := \sum_{j=1}^{|U_T^{\text{far}}|} \Big( (-1)^{j-1} \sum_{J\subseteq U_T^{\text{far}},\,|J|=j}(\Lambda(J))^{-m} \Big).
    \end{align*}
\end{theorem}

Theorem \ref{thm_net} above depends only on network properties along the shortest paths from
the searcher initial position to the furthest parts of the target. If this furthest region contains only one node, we have the following corollary:

\begin{corollary} \label{cor5}
    Under the assumptions of Theorem \ref{thm_net}, assume further that $|U_T^{\text{far}}| = 1$. Then 
        \begin{align*}
    \mathbb{E}[(T_r)^m] \sim \frac{m!}{(rp^*)^m} \quad\text{as }r\to\infty
    \end{align*}
    where
    \begin{align*}
        p^* = \frac{\Lambda(U_T^{\text{far}})}{r^{h^*}} \sim p_{U_T^{\text{far}}} \quad\text{as }r\to\infty.
    \end{align*}
    Further, $rp^*T_r$ converges in distribution to a unit rate exponential random variable,
    \begin{align*}
        rp^*T_r \to_{\textup{dist}} \textup{Exponential}(1)\quad\text{as }r\to\infty.
    \end{align*}
\end{corollary}

Note that using \eqref{ropt} on Theorem \ref{thm_net} to estimate the optimal resetting rate on a network yields the approximation $r_{\textup{opt}} \approx 0$ despite stochastic resetting having the capacity to reduce the MCT. Below we provide an alternative estimate of the optimal resetting rate on a network as well as a criterion for when inducing minimal stochastic resetting into the search process reduces the MCT. In particular, the proof of Theorem \ref{thm_net} reveals that the MCT can be described exactly by MFPTs to target nodes,
\begin{align} \label{mct_exact}
    \mathbb{E}[T_r] = \sum_{j=1}^{|U_T|} \Big( (-1)^{j-1} \sum_{J\subseteq U_T,\,|J|=j} \mathbb{E}[\tau_r(J)] \Big).
\end{align}

Supplementing a recent derivation \cite{PhysRevE.102.022115} with a higher order term, we find that MFPTs under infrequent exponentially-distributed resetting satisfy
\begin{align} \label{crit}
    \mathbb{E}[\tau_r(J)] = \mathbb{E}[\tau_0(J)] + \frac{r}{2} \Phi(J) + \frac{r^2}{2}\Big( \frac{1}{3}\mathbb{E}[\tau_0^3(J)] + \Phi(J)\mathbb{E}[\tau_0(J)] \Big) + o(r^2)
\end{align}
where
\begin{align*}
    \Phi(J) := (\mathbb{E}[\tau_0(J)])^2 - \text{Var}(\tau_0(J))
\end{align*}
and $\textup{Var}(\tau)$ denotes the variance of $\tau$. Hence by substitution of \eqref{crit} into \eqref{mct_exact},
\begin{align} \label{5352}
\begin{split}
    \mathbb{E}[T_r] &=  \mathbb{E}[T_0] + \frac{r}{2} \sum_{j=1}^{|U_T|} \Big( (-1)^{j-1} \sum_{J\subseteq U_T,\,|J|=j} \Phi(J) \Big)\\
    &\quad + \frac{r^2}{2} \sum_{j=1}^{|U_T|} \Big( (-1)^{j-1} \sum_{J\subseteq U_T,\,|J|=j} \big( \frac{1}{3}\mathbb{E}[\tau_0^3(J)] + \Phi(J) \mathbb{E}[\tau_0(J)] \big)\Big) + o(r^2).
\end{split}
\end{align}

Infrequent resetting thus reduces the MCT on a network if
\begin{align} \label{critt}
    \sum_{j=1}^{|U_T|} \Big( (-1)^{j-1} \sum_{J\subseteq U_T,\,|J|=j}\Phi(J)\Big) < 0.
\end{align}
If the target is merely one node (i.e.\ $|U_T|=1$), we recover the corresponding criterion for reducing the MFPT in \cite{PhysRevE.102.022115} which requires that the coefficient of variation of the first passage time without resetting be greater than one,
\begin{align*}
    \textup{CV}(\tau_0(U_T)) := \frac{\sqrt{\textup{Var}(\tau_0(U_T))}}{\mathbb{E}[\tau_0(U_T)]} > 1.
\end{align*}
Assuming the criterion in \eqref{critt} is satisfied, we can use the zero of the $r$-derivative of \eqref{5352} to attain an approximation of the optimal resetting rate to minimize MCTs on a network,
\begin{align*}
    r_{\textup{opt}} \approx \frac{\sum_{j=1}^{|U_T|}\Big( (-1)^j \sum_{J\subseteq U_T,\,|J|=j} \Phi(J)\Big)}{2 \sum_{j=1}^{|U_T|}\Big( (-1)^{j-1} \sum_{J\subseteq U_T,\,|J|=j} \big((1/3)\mathbb{E}[\tau_0^3(J)] + \Phi(J) \mathbb{E}[\tau_0(J)] \big)\Big)}.
\end{align*}

\subsection{Proof} \label{pf_net}
The proof of Theorem \ref{thm_net} relies on the following lemma which was proven in \cite{FPTuFSR} (see Proposition 1 in \cite{FPTuFSR}). Lemma \ref{FPT_net} itself uses that the distribution of the search time without resetting evolves according to
\begin{align} \label{t0O}
    \mathbb{P}(\tau_0(\Omega) \leq t) \sim \frac{\Lambda(\Omega)}{h(i_0,\Omega)!} t^{h(i_0,\Omega)} \quad\text{as }t\to 0^+
\end{align}
as shown in Proposition 1 of \cite{lawley2020networks}.
\begin{lemma} \label{FPT_net}
    Let $\Omega \subset \Gamma$ and assume \eqref{t0O}. Then, for integers $m\geq 1$,
    \begin{align*}
        \E[(\tau_r(\Omega))^m] \sim \frac{m!}{(rp_{\Omega})^m} \quad\text{as }r\to\infty
    \end{align*}
    and $p_{\Omega}$ in \eqref{pOm} satisfies
    \begin{align*}
        p_{\Omega} \sim \frac{\Lambda(\Omega)}{r^{h(i_0,\Omega)}}\quad\text{as }r\to\infty.
    \end{align*}
\end{lemma}

\begin{proof}[Proof of Theorem \ref{thm_net}]
By the inclusion-exclusion principle \cite{durrett2019},
\begin{align} \label{PTr_net}
\begin{split}
    \mathbb{P}(T_r > t) &= \mathbb{P}(\cup_{j\in U_T} \{\tau_r(j)>t \})\\
    &= \sum_{j=1}^{|U_T|} \Big( (-1)^{j-1} \sum_{J\subseteq U_T,\,|J|=j} \mathbb{P}(\cap_{j\in J} \{\tau_r(j)>t \}) \Big)\\
    &= \sum_{j=1}^{|U_T|} \Big( (-1)^{j-1} \sum_{J\subseteq U_T,\,|J|=j} \mathbb{P}(\tau_r(J) >t) \Big).
\end{split}
\end{align}
Substituting $t$ for $z^{1/m}$ with $m\in \mathbb{N}\backslash \{0\}$ in \eqref{PTr_net} and integrating over $z\geq 0$ gives an expression for the moments of $T_r$,
\begin{align} \label{mean}
    \mathbb{E}[(T_r)^m] = \sum_{j=1}^{|U_T|} \Big( (-1)^{j-1} \sum_{J\subseteq U_T,\,|J|=j} \mathbb{E}[(\tau_r(J))^m] \Big).
\end{align}
Lemma \ref{FPT_net} and the assumption that $|U_T|<\infty$ yields
\begin{align*}
    \lim_{r\to\infty} \mathbb{E}[r^{m(1-h^*)}T_r^m] &= \lim_{r\to\infty} \sum_{j=1}^{|U_T|} \Big( (-1)^{j-1} \sum_{J\subseteq U_T,\,|J|=j} \frac{r^{-m h^*}}{p_J^m} \mathbb{E}[(rp_J\tau_r(J))^m] \Big)\\
    &= \sum_{j=1}^{|U_T|} \Big( (-1)^{j-1} \sum_{J\subseteq U_T,\,|J|=j} \lim_{r\to\infty}\frac{r^{-m h^*}}{p_J^m} \lim_{r\to\infty}\mathbb{E}[(rp_J\tau_r(J))^m] \Big)\\
    &= \sum_{j=1}^{|U_T|} \Big( (-1)^{j-1} \sum_{J\subseteq U_T,\,|J|=j} \lim_{r\to\infty}\frac{r^{m(h(i_0,J)-h^*)}}{(\Lambda(J))^m} \\
    &\qquad \times\lim_{r\to\infty} \mathbb{E}[(rp_J\tau_r(J))^m] \Big)\\
    &= m! \sum_{j=1}^{|U_T^{\text{far}}|} \Big( (-1)^{j-1} \sum_{J\subseteq U_T^{\text{far}},\,|J|=j} \Lambda(J)^{-m} \Big).
\end{align*}
The final step uses that $h(i_0,J) < h^*$ unless $J\subseteq U_T^{\text{far}}$ in which case $h(i_0,J) = h^*$. Setting $|U_T^{\text{far}}| = 1$ immediately yields the moment convergence in Corollary \ref{cor5} and convergence in distribution follows from Theorem 30.2 in \cite{billingsley1995}.
\end{proof}

\section{Examples} \label{examples}
\noindent Below we illustrate the results in sections \ref{freq} and \ref{net} with a few examples.

\subsection{Pure diffusion in one dimension} \label{exact}
Consider a 1D Brownian searcher with diffusivity $D>0$ that resets according to an exponential random variable with rate $r>0$ to its initial position, $x_0=x_r=0$. Below we determine the MCT of the interval $[-a,b]$ with $a>0$ and $b>0$ in two distinct cases: In the first case the searcher is unconstrained on the real line and in the second case the searcher is confined to the interval $[-a,b]$ with reflecting boundary conditions. For simplicity we assume the searcher is a point particle; if instead the searcher has detection radius $R>0$ then one can study the point particle cover time of $[R-a,b-R]$.
\\ \\
\noindent \emph{Unconstrained search.} Consider first that the searcher is unconstrained on $\mathbb{R}$. It follows from the strong Markov property that the cover time of the interval $[-a,b]$ can be described by the sum of the exit time from $(-a,b)$ and the exit time from the half line with initial position $X_r(0) = a+b$ and resetting position either $a>0$ or $b>0$ depending on which interval boundary was first contacted. The mean interval exit time is known \cite{aPal2019},
\begin{align} \label{unc1}
    \E(\tau_r(\{-a,b\})) = \frac{1}{r}\left[ \frac{\sinh((a+b) \alpha)}{\sinh(a\alpha) + \sinh(b\alpha)} - 1 \right]
\end{align}
where $\alpha:=\sqrt{r/D}$. The mean half line exit time with initial position $x_0>0$ and resetting position $x_r>0$ is also known \cite{Evans2020},
\begin{align} \label{unc0}
    \E(\tau_r(\{0\}|x_0,x_r)) := \E(\inf\{t>0:X_r(t) \not>0\}) = \frac{1 - e^{-\alpha x_0}}{re^{-\alpha x_r}}.
\end{align}
Furthermore, the hitting probabilities of $x=-a$ and $x=b$ are given by \cite{aPal2019}
\begin{align}
    \xi_a &:= \P(\tau_r(\{-a\})<\tau_r(\{b\})) = \frac{\sinh(b\alpha)}{\sinh(a\alpha) + \sinh(b\alpha)},\label{uncea}\\
    \xi_b &:= 1 - \P(\tau_r(\{-a\})<\tau_r(\{b\})) = \frac{\sinh(a\alpha)}{\sinh(a\alpha) + \sinh(b\alpha)}.\label{unceb}
\end{align}

Combining \cref{unc1,uncea,unceb,unc0} yields the MCT of the interval $[-a,b]$ by a resetting Brownian searcher,
\begin{align} \label{unc_ETr}
\begin{split}
    \E(T_r) &= \E(\tau_r(\{-a,b\})) + \xi_a \E(\tau_r(\{0\}|a+b,b)) + \xi_b \E(\tau_r(\{0\}|a+b,a))\\
    &= \frac{e^{(a+b)\alpha} + e^{2a\alpha} + e^{2b\alpha} - e^{a\alpha} - e^{b\alpha} - 1}{r\left(e^{a\alpha} + e^{b\alpha}\right)}.
\end{split}
\end{align}

Notice in Figure \ref{fig1} that for finite resetting rate $r$ the MCT in \eqref{unc_ETr} is finite. This is in stark contrast to the analogous cover time problem without resetting; it is well-known that the MFPT to a point not equal to $X(0)$ is infinite and, hence, so too is the MCT of an interval of nonzero length. Observe further that the MCT in \eqref{unc_ETr} achieves a unique minimum with a positive resetting rate, though this optimal resetting rate does not have an explicit analytical solution.

Taking $r\to\infty$ in \eqref{unc_ETr} reveals that the MCT depends on the relative sizes of $a>0$ and $b>0$. For instance, if $a<b$, one can verify that the large $r$ behavior of \eqref{unc_ETr} evolves like
\begin{align} \label{etrrp}
\E(T_r) \sim \frac{1}{rp}\quad\text{as }r\to\infty
\end{align}
where $p:=\text{exp}(-b\alpha)$ is the probability of hitting the furthest part of the target before resetting, which is in agreement with Theorem \ref{thm1}. In this case it is straightforward to combine the asymptotic result in \eqref{etrrp} with \eqref{ropt} to approximate $r_{\textup{opt}} \approx 4D/b^2$. Hence, the asymptotic result in \eqref{etrrp} closely approximates the true optimum even for modest resetting rates.

We briefly address the MCT of a resetting Brownian motion on a symmetric interval, i.e.\ the case $|a|=b$. In this scenario, resetting strictly increases the MFPT to escape the interval. Yet repeating the analysis above with $|a|=b$ yields a range of resetting rates, $r\in(0,r^*]$, for which the MCT is reduced. Moreover, one finds the same asymptotic result as in \eqref{etrrp} but with a prefactor of $3/2$,
\begin{align} \label{32}
    \E(T_r) \sim \frac{3}{2} \frac{1}{rp}\quad\text{as }r\to\infty.
\end{align}
This prefactor is explained by the success probability in \eqref{main} having frequent resetting behavior determined not merely by one endpoint of the target interval but by both ends. The same phenomenon can be observed in the case of a constrained search, which we consider next.
\\ \\
\noindent \emph{Constrained search.} Consider the same process as before on the interval $[-a,b]$ but now with reflecting boundary conditions. The cover time of the interval can now be described by the sum of (i) the hitting time of $x\in\{-a,b\}$ with resetting to $x=0$ and (ii) the hitting time of $x=0$ with a reflecting boundary condition at  $x=a+b$ with initial position $X_r(0)=a+b$ and resetting to either $a>0$ or $b>0$ depending on which interval boundary was first hit.
\begin{figure}[h!]
  \centering
\includegraphics[width=\textwidth]{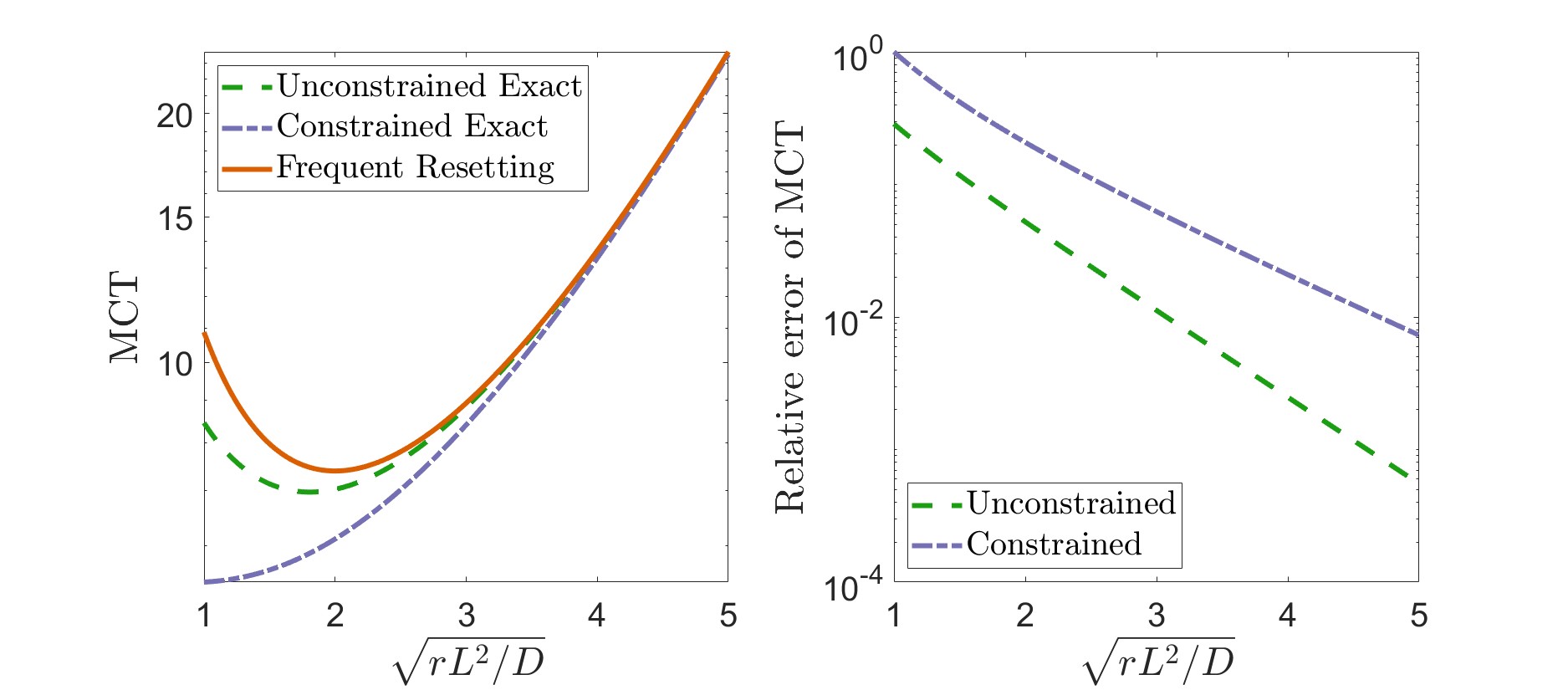}
  \caption{MCTs of the interval $[-1,2]$ for a resetting Brownian motion with diffusivity $D=1$ and $L=2$ denoting the distance of the furthest target point from the searcher initial and resetting position. The left plot illustrates \eqref{etrrp} (solid), \eqref{unc_ETr} (dashed), and \eqref{conc_ETr} (dash-dotted). The right plot illustrates the relative error of the frequent resetting result with respect to the exact unconstrained and constrained expressions. Resetting is exponentially distributed with rate $r>0$. See section \ref{exact} for details.}
  \label{fig1}
\end{figure}
The mean of the first term was computed in the previous section. Below we determine the mean time of the second term.

Suppose the searcher hits $x=b$ before $x=-a$. Denote the survival probability of this process with resetting rate $r>0$ by $Q_r(x,t)$,
\begin{align*}
    Q_r(x,t) := \mathbb{P}(\tau_r(\{-a\}) > t | X_r(0) = x).
\end{align*}
The backward Fokker-Planck equation of the survival probability without resetting is given by
\begin{align} \label{bfp}
    \frac{\partial Q_0(x,t)}{\partial t} = D\frac{\partial^2 Q_0(x,t)}{\partial x^2}
\end{align}
with boundary conditions $Q_0(-a,t) = 0$ and $\frac{\partial Q_0(x,t)}{\partial x}\big|_{x=b} = 0$. Denote the Laplace transform in time of $Q_r(x,t)$ by $\tilde{Q}_r(x,s)$. The Laplace transform of \eqref{bfp} satisfies the ordinary differential equation
\begin{align*}
    s\tilde{Q}_0(x,s) - 1 = D\frac{\partial^2 \tilde{Q}_0(x,s)}{\partial x^2}
\end{align*}
with the same boundary conditions as $Q_0$. We substitute the solution of this boundary value problem into the following known expression of $\tilde{Q}_r(x,s)$ in terms of $\tilde{Q}_0(x,s)$ (which can be derived by way of a renewal equation, see \cite{Evans2020}),
\begin{align} \label{qr}
    \tilde{Q}_r(x_0,s) = \frac{\tilde{Q}_0(x_0,r+s)}{1-r\tilde{Q}_0(x_r,r+s)},
\end{align}
and finally evaluate \eqref{qr} at $s=0$ with $x_0 = b$ and $x_r = 0$ to find that
\begin{align} \label{ET1}
    \E(\tau_r(\{-a\}|b,0)) = \frac{(e^{(a+b)\alpha} - 1)^2}{re^{a\alpha}(e^{2b\alpha} + 1)}.
\end{align}
By the analogous argument wherein the searcher hits $x=-a$ before $x=b$, we find
\begin{align} \label{ET2}
    \E(\tau_r(\{b\}|-a,0)) = \frac{(e^{(a+b)\alpha} - 1)^2}{re^{-b\alpha}(e^{2a\alpha} + 1)}.
\end{align}

Combining \cref{ET1,ET2} with the mean time of the first term and the relevant hitting probabilities computed in the previous system yields the MCT of the closed interval $[-a,b]$ by a resetting Brownian searcher,
\begin{align} \label{conc_ETr}
\begin{split}
    \E(T_r) &= \frac{1}{r}\left[\frac{\sinh((a+b) \alpha)}{\sinh(a\alpha) + \sinh(b\alpha)} - 1 \right] +\dots \\ &\qquad \frac{\xi_a}{r}\left[\frac{(e^{(a+b)\alpha} - 1)^2}{e^{b\alpha}(e^{2a\alpha} + 1)} \right] + \frac{\xi_b}{r}\left[ \frac{(e^{(a+b)\alpha} - 1)^2}{e^{a\alpha}(e^{2b\alpha} + 1)} \right].
\end{split}
\end{align}
As in the unconstrained case, taking $r\to\infty$ reveals that the MCT depends on the relative sizes of $a>0$ and $b>0$. If $a<b$, the large $r$ behavior of \eqref{conc_ETr} behaves like
\begin{align*}
\E(T_r) \sim \frac{1}{rp}\quad\text{as }r\to\infty
\end{align*}
where $p:=\text{exp}(-b\alpha)$ is again the probability of hitting the further boundary point before a resetting event. We illustrate this behavior in Figure \ref{fig1}. Similar to the unconstrained case, the asymptotic result here closely approximates the true optimal resetting rate.

\subsection{Run-and-tumble particle in one dimension} \label{rtp}
Consider a 1D run-and-tumble particle (RTP) with velocity $v>0$ and switching rate $\gamma>0$ who resets according to an exponential random variable with rate $r>0$ to its initial position, $x_0=x_r=0$. Below we determine the unconstrained and constrained MCTs of the interval $[-a,a]$. For simplicity we again assume the searcher is a point particle.
\\ \\
\noindent \emph{Unconstrained search.} By the strong Markov property, the cover time of $[-a,a]$ can be described by the sum of the exit time from $(-a,a)$ and the exit time from the half line with initial position $X_r(0) = 2a$ and resetting position $a>0$. The mean interval exit time is known \cite{RTP_II},
\begin{align} \label{unc1_rtp}
    \E(\tau_r(\{-a,a\})) = \frac{1}{r} \left( \cosh ac_r + \sqrt{\frac{r}{r+2\gamma}}\sinh ac_r - 1 \right)
\end{align}
where $c_r := \sqrt{r(r+2\gamma)/v^2}$. The mean half line exit time is also known \cite{Evans_RTP},
\begin{align} \label{unc0_rtp}
    \E(\tau_r(\{0\}|a,a)) = \frac{1}{r}\left( \frac{2\gamma e^{ac_r}}{r + 2\gamma - vc_r} - 1 \right).
\end{align}
The sum of \eqref{unc1_rtp} and \eqref{unc0_rtp} yields the MCT of the interval $[-a,a]$ by a resetting RTP,
\begin{align} \label{unc_ETr_rtp}
    \E(T_r) = \frac{1}{r}\left( \cosh ac_r + \sqrt{\frac{r}{r+2\gamma}}\sinh ac_r + \frac{2\gamma e^{ac_r}}{r + 2\gamma - vc_r} - 2 \right).
\end{align}
Taking $r\to\infty$ in \eqref{unc_ETr_rtp}, one finds that the large $r$ behavior of \eqref{unc_ETr} evolves according to \eqref{32} where now $p:=\textup{exp}(-a(\gamma+r)/v)/2$ is the probability of hitting $x=a$ or $x=-a$ before resetting.
\begin{figure}[h!]
  \centering
  \includegraphics[width=\textwidth]{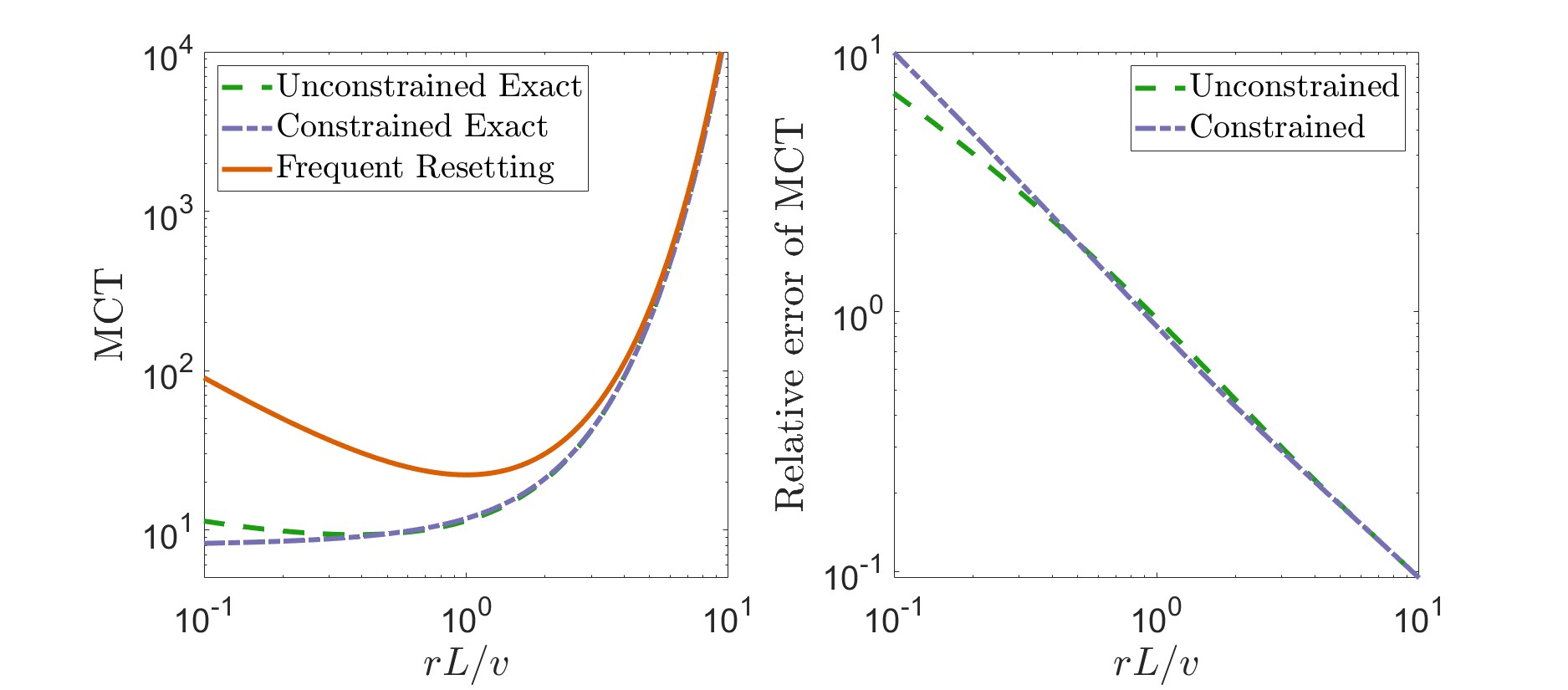}
  \caption{MCTs of the interval $[-1,1]$ for a resetting RTP with velocity $v=1$, switching rate $\gamma=1$, and $L=1$ denoting the distance of the furthest target points from the searcher initial and resetting position. The left plot illustrates \eqref{32} (solid), \eqref{unc_ETr_rtp} (dashed), and \eqref{con_ETr_rtp} (dash-dotted). The right plot illustrates the relative error of the frequent resetting result with respect to the exact unconstrained and constrained expressions. Resetting is exponentially distributed with rate $r>0$. See section \ref{rtp} for details.}
  \label{fig_RTP}
\end{figure}
While there does not exist an analytical expression of the optimal resetting rate from \eqref{unc_ETr_rtp}, combining \eqref{32} with \eqref{ropt} yields the approximation $r_{\textup{opt}} \approx v/a$. As in the symmetric case of subsection \ref{exact}, resetting strictly increases the MFPT to escape the interval, but there exists a range of resetting rates for which the MCT is reduced.
\\ \\
\noindent \emph{Constrained search.} Consider the same process as before but on the interval $[-a,a]$ with reflecting boundary conditions. Now the cover time can be described by the sum of (i) the hitting time of $x\in\{-a,a\}$ with resetting to $x=0$ and (ii) the hitting time of $x=-a$ with a reflecting boundary and initial condition at $x=a$ and resetting to $x=0$. The mean of the first term was computed in the previous section. Below we determine the mean time of the second term.

The Laplace transform of the exit time density of the half-open interval $(-a,a]$ for a non-resetting RTP with $x_0\in(-a,a]$ is known \cite{Angelani_2023},
\begin{align}
    \tilde{\phi}(s,x_0) = \frac{(s+2\gamma)\cosh(c_s(a-x_0))}{(s+2\gamma)\cosh(2ac_s) + vc_s\sinh(2ac_s)}.
\end{align}
Since the exit time density is related to the survival probability by
\begin{align} \label{phi_tx0}
    \phi(t,x_0) = -\frac{\partial}{\partial t} Q_0(t,x_0),
\end{align}
taking the Laplace transform of \eqref{phi_tx0} and rearranging terms yields
\begin{align}
    \tilde{Q}_0(s,x_0) = \frac{1-\tilde{\phi}(s,x_0)}{s}.
\end{align}
Thus, evaluating \eqref{qr} at $s=0$ with $x_0 = a$ and $x_r = 0$, we find
\begin{align} \label{con0_rtp}
    \E(\tau_r(\{-a\}|a,0)) = \frac{2 \left(vc_r \sinh ac_r + r\cosh ac_r \right) \tanh ac_r}{rvc_r}
\end{align}

The sum of \eqref{unc1_rtp} and \eqref{con0_rtp} thus yields the MCT of the interval $[-a,a]$ with reflecting boundary conditions by a resetting RTP,
\begin{align} \label{con_ETr_rtp}
    \mathbb{E}(T_r) = \frac{vc_r\cosh ac_r + \sinh ac_r \left(2vc_r
   \tanh ac_r + 3r\right) - vc_r}{rvc_r}.
\end{align}
As expected, taking $r\to\infty$ in \eqref{con_ETr_rtp} yields \eqref{32} where $p:=\textup{exp}(-a(\gamma+r)/v)/2$. We illustrate this behavior in Figure \ref{fig_RTP}.

\subsection{Pure diffusion in $\mathbb{R}^d$} \label{ball}
Here we consider the $d$-dimensional analogue to the one-dimensional unconstrained search in section \ref{exact}. Consider a Brownian searcher with fixed detection radius $R>0$ and diffusivity $D>0$ in $\mathbb{R}^d$. Suppose the target is a ball with radius $a>0$ centered at the initial (resetting) position of the searcher $U_T = B(X_r(0),a)$. Hence $U_T^{\text{far}}$ is the sphere of radius $a>0$ centered at $X_r(0)$.

The first hitting time of a purely diffusive search process without resetting to $\Omega\subseteq M$ has the short time behavior \cite{lawley2020uni},
\begin{align} \label{limt}
    \lim_{t\to 0^+} t \ln \mathbb{P}(\tau_0(\Omega) \leq t) = - L^2(X_r(0),\Omega)/4D < 0
\end{align}
where $L$ is the standard Euclidean metric.

Assume resetting events are exponentially distributed with rate $r>0$. That is, $Y$ in \eqref{sig} is a unit rate exponential random variable. Then Theorem 2 in \cite{FPTuFSR} combined with \eqref{limt} implies that Theorem \ref{thm1} holds with $p_{B(x^*,R)}$ satisfying
\begin{align} \label{lnp}
    \ln p_{B(x^*,R)} \sim -\sqrt{(a-R)^2r/4D}\quad\text{as }r\to\infty.
\end{align}
Notice that \eqref{lnp} is independent of the spatial dimension and depends only weakly on the detection radius. Moreover, the geometry of the target is relevant only up to the shortest distance to it from the initial position of the searcher. Finally, combining \eqref{lnp} with \eqref{ropt} we approximate the optimal resetting rate to be $r_{\textup{opt}} \approx 16D/(a-R)^2$.

\subsection{Pure diffusion on a $d$-dimensional torus} \label{torus}
Consider a Brownian searcher with fixed detection radius $R>0$ and diffusivity $D>0$ on a $d$-dimensional torus with diameter $\ell=\sup_{y_1,y_2 \in M} \|y_1-y_2\| \in (R,\infty)$ where
\begin{align*}
    M = [0,\ell \sqrt{d})^d \subset \mathbb{R}^d
\end{align*}
with periodic boundary conditions. Suppose the target is the entire torus, $U_T = M$. Let the searcher start from and reset to the torus ``origin'' given by
\begin{align*}
    X_r(0) = (0,0,\dots,0)\subset M.
\end{align*}
Then $U_T^{\text{far}} = \{(\ell\sqrt{d}/2,\dots,\ell\sqrt{d}/2)\}$.

As in subsection \ref{ball}, the first hitting time of the search process without resetting to $\Omega\subseteq M$ in the case of pure diffusion has the short time behavior \eqref{limt} where $L$ is the standard Euclidean metric. Assuming resetting events are exponentially distributed with rate $r>0$, Theorem 2 in \cite{FPTuFSR} combined with \eqref{limt} implies that Theorem \ref{thm1} holds with $p_{B(x^*,R)}$ satisfying
\begin{align} \label{lnpp}
    \ln p_{B(x^*,R)} \sim -\sqrt{(\ell-2R)^2r/4D}\quad\text{as }r\to\infty.
\end{align}
The result in \eqref{lnpp} shares similarly weak dependencies on system parameters as in \eqref{lnp} on the $d$-dimensional sphere. Moreover combining \eqref{lnpp} with \eqref{ropt} we approximate the optimal resetting rate to be $r_{\textup{opt}} \approx 16D/(\ell-2R)^2$.

\subsection{Subdiffusion}
Consider a searcher with fixed detection radius $R>0$ that subdiffuses on $M\subseteq \mathbb{R}^d$ between stochastic resets, which can be modeled by a fractional Fokker-Planck equation \cite{anom_FP},
\begin{align} \label{subdiff}
    \frac{\partial }{\partial t} p_{\alpha} = \prescript{}{0}{D}_t^{1-\alpha} K_{\alpha} \Delta p_{\alpha}, \quad t>0,x\in M.
\end{align}
In \eqref{subdiff}, $t>0$ is the time elapsed since the most recent reset event, $p_{\alpha} = p_{\alpha}(x,t)$ denotes the probability density of the searcher, $K_{\alpha}$ is the generalized diffusivity (with units $(\text{length})^2(\text{time})^{-\alpha}$), and $\prescript{}{0}{D}_t^{1-\alpha}$ is the Riemann-Liouville fractional derivative \cite{samko1993fractional},
\begin{align*} 
    \prescript{}{0}{D}_t^{1-\alpha} f(t) = \frac{1}{\Gamma(\alpha)} \frac{d}{dt} \int_0^t \frac{f(s)}{(t-s)^{1-\alpha}}\, \textup{d}s.
\end{align*}
where $\Gamma(\alpha)$ is the gamma function. Hence its mean-squared displacement (without resetting) evolves in time according to a sublinear power law $t^{\alpha}$ for $\alpha\in(0,1)$. 

One can obtain paths of subdiffusive searchers whose probability densities satisfy \eqref{subdiff} by way of a random time change, or subordination, of diffusive search paths \cite{magdziarz2016}. In particular, denote by $X = \{X(s)\}_{s\geq 0}$ the path of a diffusive searcher,
\begin{align*}
    \textup{d}X(s) = \sqrt{2D}\,\textup{d}W(s),
\end{align*}
where $W=\{W(s)\}_{s\geq 0}$ is a standard Brownian motion. Then a corresponding subdiffusive path is given by
\begin{align} \label{XtoY}
    Y(t) = X(S(t)),\quad t\geq 0
\end{align}
where $\{S(t)\}_{t\geq 0}$ is an inverse $\alpha$-stable subordinator independent of $X$ \cite{magdziarz2016}.

Within this framework, Corollary 3 in \cite{lawley2020sub} indicates that if first passage times of $X$ satisfy \eqref{limt} then those of $Y$ defined by \eqref{XtoY} satisfy
\begin{align} \label{tbetaln}
    \lim_{t\to 0^+} t^{\beta} \ln \mathbb{P}(\tau_0(\Omega)\leq t) = -(2-\alpha) \alpha^{\beta} \bigg( \frac{-L^2(U_0,\Omega)}{4D} \bigg)^{\beta/\alpha} < 0
\end{align}
where $\beta = \alpha/(2-\alpha)$. Assuming resetting events are exponentially distributed with rate $r>0$, Theorem 4 in \cite{FPTuFSR} combined with \eqref{tbetaln} implies that Theorem \ref{thm1} holds with $p_{B(x^*,R)}$ satisfying
\begin{align} \label{lnp_sub}
    \ln p_{B(x^*,R)} \sim -\gamma r^{\beta/(\beta+1)} \quad\text{as }r\to\infty
\end{align}
where
\begin{align*}
    \gamma = \frac{\beta+1}{\beta^{\beta/(\beta+1)}} C^{1/(\beta+1)},\quad C = (2-\alpha) \alpha^{\beta} \bigg( \frac{-L^2(U_0,U_T^{\text{far}})}{4D} \bigg)^{\beta/\alpha}.
\end{align*}
Hence combining \eqref{lnp_sub} with \eqref{ropt} we approximate the optimal resetting rate to be $r_{\textup{opt}} \approx ((\beta+1)/\gamma\beta)^{(\beta+1)/\beta}$.

We briefly note that, like subdiffusion, superdiffusive L\'evy flights can also be described by subordinated Brownian motion \cite{lawley2021super} and thus one can use a similar framework to understand the corresponding cover time behavior.

\begin{figure}[h!]
    \centering
    \includegraphics[clip, trim=.8cm 4cm 13.8cm 4cm, width=.36\textwidth]{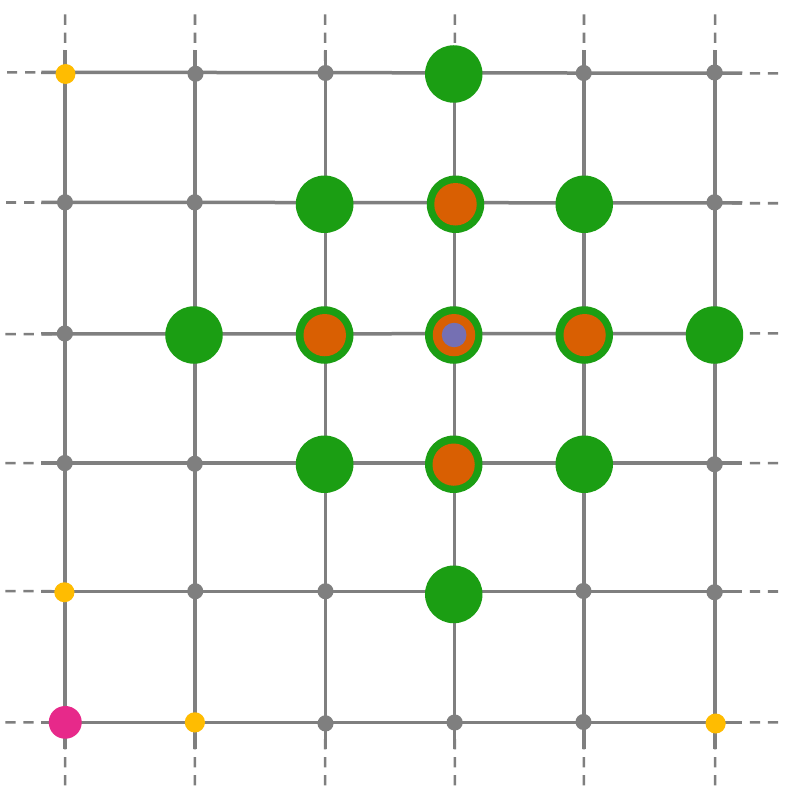} \hspace{.2cm}
    \includegraphics[width=.46\textwidth]{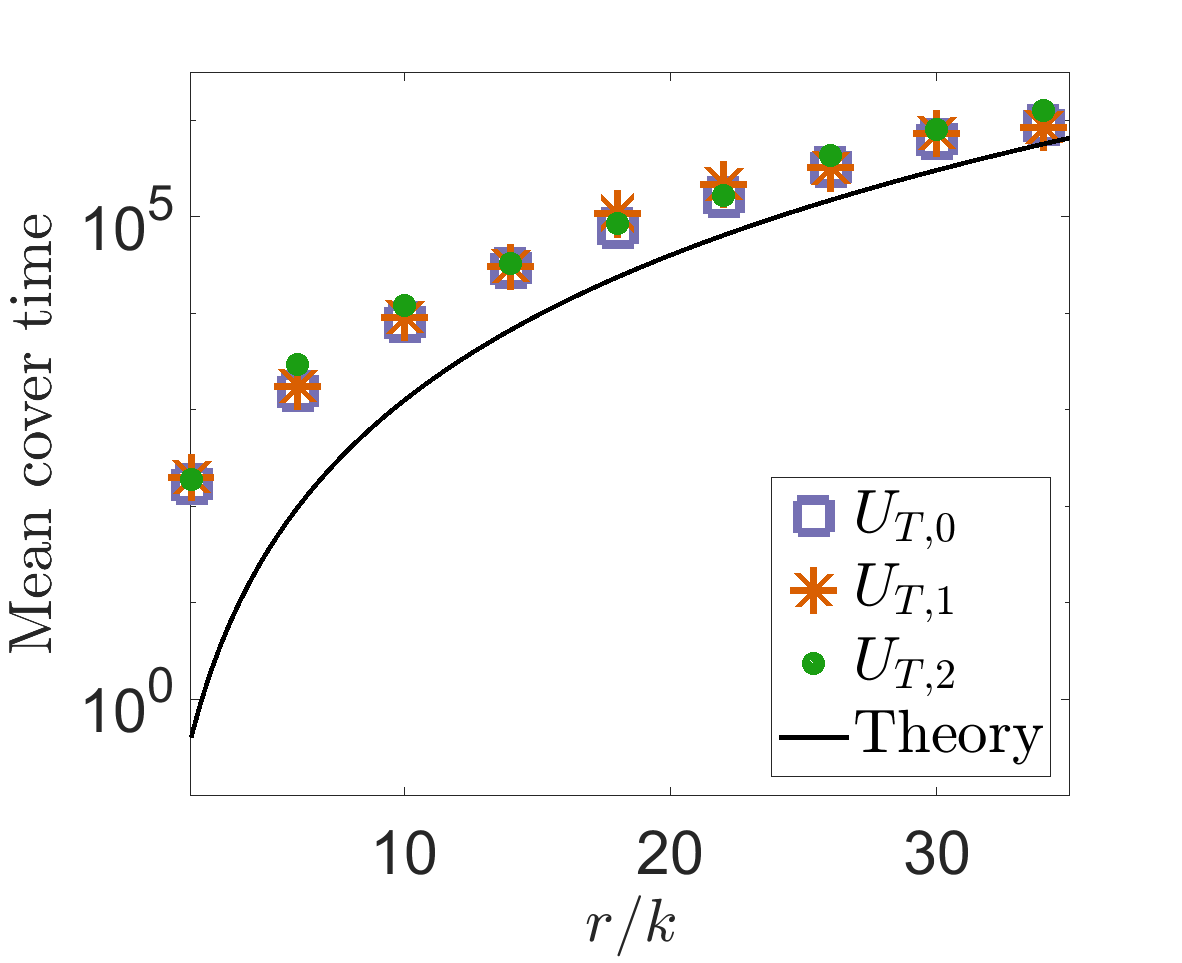}
    \caption{Left: The initial node of the searcher in pink, which can take its first step to any of the four yellow nodes; the simulated nested target sets $U_{T,j}$ in \eqref{UTj} are for $j=0$ (purple), $j=1$ (orange), and $j=2$ (green). Right: The theoretical MCT in Corollary \ref{cor5} to simulations of that for the depicted nested target sets. Each point is the average of 20 simulations with jump rates $k=1$. See section \ref{markov} for details.}
    \label{net_fig}
\end{figure}

\subsection{Markovian search on a $g$-dimensional periodic lattice} \label{markov}
Consider a $g$-dimensional square lattice $I$ with periodic boundary conditions, i.e.\ a square lattice wrapped around the $g$-dimensional torus. Suppose the searcher starts from and resets to $i_0\in I$. Suppose further that the number of nodes along each edge is even and given by $2\ell$ for some $\ell\in\mathbb{N}$. There thus exists a unique furthest node $i_1\in I$ from $i_0\in I$. Let the target set $U_T\subseteq I$ be any set containing $i_1$ so that $U_T^{\textup{far}}$ in \eqref{UTnet} satisfies
\begin{align*} 
    U_T^{\textup{far}} = i_1 \subseteq U_T \subseteq I.
\end{align*}
Hence $h^*$ as in \eqref{dstar} satisfies $h^* = h(i_0,i_1) = g\ell$.

If all jumps in the lattice occur with the same rate $k>0$, then $\Lambda$ is as in \eqref{Lamb},
\begin{align*}
    \Lambda(U_T^{\textup{far}}) = k^{ h^*} |S(i_0,i_1)|,
\end{align*}
with $S$ as in \eqref{S}, the cardinality of which is given by
\begin{align*}
    |S(i_0,i_1)| = 2^g{g\ell\choose \ell,\ell,\dots,\ell} = 2^g\frac{(g\ell)!}{(\ell!)^g}.
\end{align*}
\noindent Hence, $p^*$ in Corollary \ref{cor5} satisfies 
\begin{align} \label{ph}
    p^* = 2^g\frac{(g\ell)!}{(\ell!)^g} \Big(\frac{k}{r}\Big)^{g\ell} \sim \sqrt{2\pi g\ell} \Big(\frac{2}{\pi\ell}\Big)^{g/2} \Big(\frac{gk}{r}\Big)^{g\ell} \quad\text{as }\ell\to\infty
\end{align}
by way of Stirling's approximation. In the case that $g=2$, \eqref{ph} simplifies to
\begin{align*}
    p^* = \frac{4^{1+\ell}}{\sqrt{\pi\ell}} \Big(\frac{k}{r}\Big)^{2\ell} \quad\text{as }\ell\to\infty.
\end{align*}

Now we compare our analysis to numerical simulations on the two dimensional square lattice with periodic boundary conditions. We take $2\ell = 6$, hence $|I|=(2\ell)^2 = 36$. We use the Gillespie algorithm on the position of the searcher $\{X(t)\}_{t\geq 0}$ on the state space $I$ with jumping rate $k=1$ between each pair of connected nodes.

In Figure \ref{net_fig}, we plot the results of stochastic simulations on this network with nested target sets,
\begin{align} \label{UTj}
    U_{T,j} = \{i\in I : h(i,U_T^{\textup{far}}) \leq j\}, \quad j=\{0,1,2\},
\end{align}
where
\begin{align*}
    U_T^{\textup{far}} = U_{T,0},\quad U_{T,j} \subseteq U_{T,j+1}.
\end{align*}
Since these target sets contain $U_T^{\textup{far}}$ in \eqref{UTnet}, Corollary \ref{cor5} implies that their cover times converge in the frequent resetting limit, as evident in Figure \ref{net_fig}. Moreover, the similarity in simulated MCTs across target sets is suggestive of the cover time being governed by the search time to the furthest target node.
\section{Discussion}
In this work we estimated all the moments of cover times of a wide range of stochastic search processes in $d$-dimensional continuous space and on an arbitrary discrete network in the frequent resetting limit. These results depend only on the resetting rate, which is well-defined for a large class of resetting time distributions, and the so-called success probability corresponding to the furthest target point(s). These probabilities need not be known exactly; their asymptotic behavior in the frequent resetting limit suffices. Recent work on success probabilities computes these values for a variety of stochastic search and resetting processes \cite{linn2024hitting}.

While these results indeed depend on only a few system parameters, the dependence is far from trivial. Moreover, because knowledge of only a few parameters is necessary, the cover time moments with frequent resetting are indifferent to changes in the domain away from geodesic paths as well as additional targets that are strictly closer than the furthest target regions. This feature is reminiscent of work by Chupeau et al that shows cover times of a non-resetting searcher are dominated by FPTs to hard-to-reach targets \cite{CT_chupeau,Barkai2015}. In particular, they establish a universal distribution for the cover time rescaled by an average FPT whose greatest contributions are the FPTs to distant targets. An important distinction between these works, however, is that the frequently resetting searcher will take near-geodesic paths from its initial position to far target regions, whereas the non-resetting searcher in \cite{CT_chupeau} forgets its initial condition entirely. Both results nonetheless are related by the necessity of rare events to cover the target region.

We finally note that computing the MCT exactly in complicated domains can be involved or analytically intractable, and simulation is more cumbersome than the analogous first passage time problem. Through the main results herein, we establish a means of approximating the optimal resetting rate to minimize the MCT. We also establish a criterion for when endowing a discrete space search process with resetting reduces the MCT. Importantly, these results utilize and build upon a growing body of work concerning the role of resetting in reducing mean search times in a variety of systems \cite{evans_majumdar_2011, PhysRevE.102.022115, GhoshPK_2023_JCP}. That said, determining the role of dimension, domain geometry, and target geometry on cover times would be an interesting and challenging potential avenue for future work.
\newpage
\section{Acknowledgments}
\noindent SL was supported by the National Science Foundation (Grant No.\,2139322). SDL was supported by the National Science Foundation (CAREER DMS-1944574 and DMS-2325258). The authors thank Snehesh Das (UCLA) and Shreyas Waghe (U Mass Amherst) for useful discussions pertaining to section \ref{rtp}. 

\section{Data Availability} \noindent Data sharing is not applicable to this article as no new data were created or analyzed in this study.

\bibliography{lib.bib}
\bibliographystyle{hunsrt}
\end{document}